\documentclass[11pt, twoside, a4paper]{article}
\usepackage{amsfonts}
\usepackage{amstext}
\usepackage{amsthm}
\usepackage{amsmath}
\usepackage{eufrak}
\usepackage{mathrsfs}
\usepackage{amssymb}
\usepackage{latexsym}
\usepackage[latin1]{inputenc}

\newtheorem{teo}{Theorem}[section]

\newtheorem{prop}{Proposition}[section]
\newtheorem{defi}{Definition}[section]

\newtheorem{cor}{Corollary}[section]
\newtheorem{exe}{Exemple}[section]
\newtheorem{obs}{Remark}[section]

\newcommand{\F}{{\mathcal{F}}}
\begin{document}

\newtheorem{theorem}{Theorem}
\newtheorem{proposition}[theorem]{Proposition}
\newtheorem{lemma}[theorem]{Lemma}

\newtheorem{definition}{Definition}

\title{Poincaré problem for divisors
invariant by one-dimensional foliations on smooth algebraic variety
}
\author{
Maurício Corrêa JR\\
\small{Departamento de Matem\'atica}\\
\small{Universidade Federal de Minas Gerais}\\
\small{30123-970 Belo Horizonte - MG, Brasil}\\
\small{\texttt{mauriciojr@ufmg.br}}}
\date{\today}
\maketitle
\begin{abstract}
In this paper we consider the question of bounding the degree of an
divisor $D$ invariant by a $\F$ holomorphic foliation, without
rational first integral, on smooth algebraic variety $X$ in terms of
degree of $\F$ and some invariants of $D$ and $X$. Particularly, if
$\F$ is a foliation of degree $d$ on $\mathbb{P}_{\mathbb{C}}^2$,
whose the number of invariants curves is greater that ${k+2\choose
k}$, we show that there exist a number $\mathcal{M}(d,k)$ such that
if $k>\mathcal{M}(d,k),$ then $\F$ admits a rational first integral
of degree $\leq k$. Moreover, there exist a number
$\mathscr{G}(d,k)$, such that if $\F$ has an algebraic solution of
degree $k$ and genus smaller than
 $\mathscr{G}(d,k)$, then it has a rational first integral of
degree $\leq k$.
\end{abstract}

\section{Introduction}
Henri Poincaré studied in \cite{HP} the problem which, in the modern
terminology, says: \emph{"Is it possible to decide if a holomorphic
foliation $\F$ on the complex projective plane
$\mathbb{P}_{\mathbb{C}}^2$ admits a rational first integral ?"}
Poincaré observed that in order to solve this problem is sufficient
to find a bound for the degree of the generic leaf of $\F$. In
general, this is not possible, but doing some hypothesis we obtain
an affirmative answer for this problem, which nowadays is known as
$\emph{Poincaré Problem}$. Many mathematicians come treating this
problem and some of its generalizations, see for instance the papers
of  Cerveau $\&$ Lins Neto \cite{CN}, Carnicer \cite{C}, Soares
\cite{S}, Brunella $\&$ Mendes \cite{B-M}, Esteves $\&$ Kleiman
\cite{E-K}, V. Cavalier $\&$ D. Lehmann \cite{C-L}  and Zamora
\cite{Z}.

Other researcher that treated this type of problem was P. Painlevé,
more or less at the same time of Poincaré problem, which in
\cite{PP} asked the following question: \emph{"Is it possible to
recognize the genus of the general solution of an algebraic
differential equation in two variables which has a rational first
integral?}" In \cite{N} Lins Neto  has constructed families of
foliations with fixed degree and local analytic type of the
singularities where foliations with rational first integral of
arbitrarily large degree appear. Therefore this families  show that
Poincaré and Painlevé questions have a negative answer. In the same
paper Lins Neto raised the the following question: \emph{"Given
$d\geq2$, is there $M(d)\in \mathbb{N}$,such that if a foliation of
degree d has an algebraic solution of degree greater than or equal
to $M(d)$, then it has a rational first integral?"} J.Moulin
Ollagnier showed in \cite{O} that whend $d=2$ this question has a
negative answer, he exhibited a countable family of Lotka-Volterra
foliations given by
$$
SLV(\ell)=x(y/2 + z)\frac{\partial}{\partial x}+y(2z +
x)\frac{\partial}{\partial
y}+z\left(y-\frac{2\ell+1}{2\ell-1}x\right)\frac{\partial}{\partial
z}
$$
without rational first integrals such that has an irreducible
algebraic solution of degree $2\ell$.

Let $\F$ be a foliation on $\mathbb{P}_{\mathbb{C}}^n$ of degree
$d\geq2$ and $\mathcal{V}$ a hypersurface $\F$-invariante of degree
$k$. In this paper, using the \emph{extatic divisor}, we show that
if  the number of invariants hypersurfaces, of degree $k$, is
greater than $ {n+k\choose k}$ then there exist a number
$\mathcal{M}(d,k)$ such that if
$$k>\mathcal{M}(d,k),$$ then $\F$ admits a rational first
integral, see corollary \ref{pn}.

We raise the following question: \emph{"Given $d\geq2$, is there
$\mathscr{G}(d,k)$, such that if a foliation of degree $d$ has an
algebraic solution of degree $k$ and genus  smaller than to
$\mathscr{G}(d,k)$, then it has a rational first integral?"} We will
show that this question has a positive answer, see theorem \ref{gen}
.

If $\F$ is a holomorphic one-dimensional foliation on  algebraic
variety $X$, then a rational first integral for $\F$ is a rational
map $\Theta:X \dashrightarrow Y$, where $Y$ is an algebraic variety,
such that the fibers of $\Theta$ are $\F$-invariant.
Using a concept of degree of foliations and divisors  we will proof
the following result.

\begin{theorem}\label{teo}
Let $\F$ be a one-dimensional foliation on smooth algebraic variety
$X$ and $D$ a effective divisor $\F$-invariant. Suppose that $\F$
does not admit rational first integral, then:

$$deg(D)\cdot\left[\mathscr{N}(\F,|D|)-h^0(X,\mathscr{O}(D))\right]\leq
[deg(\F)-deg(X)]\cdot\displaystyle{h^0(X,\mathscr{O}(D)) \choose
2},$$
 where $\mathscr{N}(\F,|D|)$  is the number  of divisors $\F$-invariante
contained on the  linear system $H^0(X,\mathscr{O}(D))$ and
$h^0(X,\mathscr{O}(D))=dim_{\mathbb{C}}H^0(X,\mathscr{O}(D))$.
\end{theorem}

\begin{cor}\label{poin}
Let $\F$ be a one-dimensional foliation on algebraic variety $X$ and
$D$ a effective divisor $\F$-invariant. Suppose that $\F$ does not
admit rational first integral and that $\mathscr{N}(\F,|D|)>
h^0(X,\mathscr{O}(D))$, then:

$$deg(D)\leq\frac{
[deg(\F)-deg(X)]\cdot\displaystyle{h^0(X,\mathscr{O}(D)) \choose
2}}{\mathscr{N}(\F,|D|)-h^0(X,\mathscr{O}(D))},$$
 where $\mathscr{N}(\F,|D|)$  is the number  of divisors $\F$-invariante
contained on the  linear system $H^0(X,\mathscr{O}(D))$ and
$h^0(X,\mathscr{O}(D))=dim_{\mathbb{C}}H^0(X,\mathscr{O}(D))$.
\end{cor}

\begin{exe}
Let $X$ be a Abelian varietie of dimension $n$ and $D$ a effective
divisor invariant by a holomorphic foliation $\F$ on $X$. If
$\mathscr{N}(\F,|D|)> \frac{D^n}{n!}$ and  $\F$ does not admit a
rational first integral then
$$deg(D)\leq\frac{
[deg(\F)-deg(X)]\cdot\displaystyle{\frac{D^n}{n!} \choose
2}}{\mathscr{N}(\F,|D|)-\frac{D^n}{n!}}.$$
 Indeed, follows
from Kodaira-Nakano Vanishing Theorem that
$$h^p(X,\mathscr{O}(D))=h^p(X,\Omega^n(D))=0, \ \ \ \ \ p>0,$$
and hence the holomorphic Euler characteristic
$\chi(D)=h^0(X,\mathscr{O}(D))$. On the other hand, we have by
Riemann-Roch theorem that $\chi(D)=\frac{D^n}{n!}$. Now, the
affirmed follows from corollary  \ref{poin}.
\end{exe}

If we suppose that $\mathscr{N}(\F,|D|)> h^0(X,\mathscr{O}(D))$,
follows from corollary \ref{poin} that there exist a number
$\mathscr{M}(\F,|D|)$, such that if $\F$ possess a invariant
effective  divisor $D$, contained on the linear system
$|D|=H^0(X,\mathscr{O}(D))$, satisfying  the  condition
$$
deg(D)>\mathscr{M}(\F,|D|),
$$
then $\F$ admit rational first integral.

\begin{cor}\label{pn}
Let $\F$ be a one-dimensional foliation on
$\mathbb{P}_{\mathbb{C}}^n$ of degree $d\geq 2$ and
$\mathscr{N}(d,k)$ the number of hypersurfaces invariants by $\F$ of
degree $k$. Suppose that $\mathscr{N}(d,k)>{n+k\choose k}$ and there
exist a hypersurface invariant by $\F$ of degree $k$ such that
$$k>
\dfrac{(d-1)\cdot\displaystyle{{n+k\choose k} \choose
2}}{\mathscr{N}(d,k)-{n+k\choose k}}.$$ Then $\F$ admit a rational
first integral.

\end{cor}

\begin{obs}
In the case of a foliation  $\F$  on $\mathbb{P}_{\mathbb{C}}^2$, we
can to bound of degree of the rational first integral which this one
admit. See the Proposition 2 in \cite{V}.
\end{obs}

Let $X$ be  an algebraic surface and $D$ a curve invariant by a
foliation $\F$ on $X$. In this case, we obtain the following
inequality in terms of invariants of $X$, the virtual genus of $D$
and the degree of $\F$.

Let $D$ be a divisor on algebraic surface $X$. For the next result
we shall use the following notation
$h^i(D)=dim_{\mathbb{C}}H^i(X,\mathscr{O}(D))$, $i=0,1.$
\begin{cor}\label{cor}
Let $\F$ be a  foliation on algebraic surface $X$ and $D$ a divisor
$\F$-invariant. If $\F$ does not admit rational first integral, then
$$
2-2g(X,D)\leq
2h^1(D)-2h^0(K-D)+2\frac{(deg(\F)-deg(M))}{deg(D)}\cdot\displaystyle{h^0(D)\choose
2}+
$$
$$
 +\frac{K\cdot(K-12D)+\chi(X)}{6}-2\mathscr{N}(\F,|D|),
$$
where $g(X,D)$ is the virtual genus of $D$, $K$ is the canonical
sheave of $X$ and $\chi(X)$ is the Euler characteristic of $X$.
\end{cor}

\begin{teo}\label{gen}
Let $\F$ be a foliation on $\mathbb{P}_{\mathbb{C}}^2$, of degree
$d\geq2$, that does not admit rational first integral of degree
$\leq k$. Let $D$ be a algebraic curve, of degree $k$, invariant by
$\F$, then
$$
2-2g(D)\leq\frac{d(k^3+6k^2+11k+6)-k^3-6k^2+13k+2}{4}-2\mathscr{N}(d,k),
$$
where $g(D)$ is the virtual genus of $D$ and $\mathscr{N}(d,k)$ is
the number of curve $\F$-invariant of degree $k$.
\end{teo}

\begin{proof}
 Since
$\chi(\mathbb{P}_{\mathbb{C}}^2)=3$, $K=-3h$ and $D=kh$, where $h$
is the hyperplane class of $\mathbb{P}_{\mathbb{C}}^2$, follows that
$$
\frac{K\cdot(K-12D)+\chi(X)}{6}=6k+2.
$$
However, $K-D=-(3+k)h$ and  so $deg(K-D)=-3-k<0$, hence follows from
theorem \ref{posi} that $h^0(K-D)=0$. Moreover
$h^1(\mathbb{P}_{\mathbb{C}}^2,kh)=0$. In fact, we have that
$$
H^1(\mathbb{P}_{\mathbb{C}}^2,kh)=H^1(\mathbb{P}_{\mathbb{C}}^2,\Omega_{\mathbb{P}_{\mathbb{C}}^2}^2(kh-K))=
H^1(\mathbb{P}_{\mathbb{C}}^2,\Omega_{\mathbb{P}_{\mathbb{C}}^2}^2((k+3)h)),
$$
and applying the Kodaira-Nakano Vanishing Theorem that for $q=1$ and
$p=2$ we get
$$H^1(\mathbb{P}_{\mathbb{C}}^2,\Omega_{\mathbb{P}_{\mathbb{C}}^2}^2((k+3)h))=0.$$
Therefore, since $\F$ does not admit rational first integral of
degree $\leq k$, follows from corollary \ref{cor} and of the done
calculations that
$$
2-2g(D)\leq\frac{d(k^3+6k^2+11k+6)-k^3-6k^2+13k+2}{4}-2\mathscr{N}(d,k)
$$
\end{proof}

\begin{exe}
Let $C$ be a smooth curve invariant by a foliation $\F$ on
$\mathbb{P}_{\mathbb{C}}^2$. Under the conditions of theorem
\ref{gen} we get
$$
\chi(C)\leq
\frac{d(k^3+6k^2+11k+6)-k^3-6k^2+13k+2}{4}-2\mathscr{N}(d,k)
$$
where $\chi(C)$ is the Euler characteristic of $C$.
\end{exe}

Follows from theorem \ref{gen} that there exist a number
$\mathscr{G}(d,k)$, such that if $\F$ possess a invariant curve $C$,
of degree $k$, which satisfies the following condition
$$
g(C)<\mathscr{G}(d,k),
$$
then $\F$ admit rational first integral of degree $\leq k$.

\section{The degree of divisors and holomorphic foliations}
Let $(X,\varpi)$ be a Kähler manifold where $\varpi$ is the Kähler
form. The degree of holomorphic vector bundle $\mathrm{E}$ on $X$
related to structure induced be  $\varpi$ is defined by
$$
deg_{\varpi}(\mathrm{E})= \int\limits_{X}c_1(\mathrm{E})\wedge
\varpi^{n-1}.
$$

\begin{teo}\label{posi}\cite{K}
Let $L$ be a line bundle on Kahler manifold $(X,\varpi)$. Then:
\begin{itemize}

  \item [i)] If $deg_{\varpi}(L)<0$, then $H^{0}(X,\mathcal{O}(L))=\{0\}$.

  \item [ii)] If $deg_{\varpi}(L)=0$ and $s\in H^{0}(X,\mathcal{O}(L))$,
  with $s\neq0$, then $s(p)\neq 0$ for all $p\in X$.

\end{itemize}

\end{teo}

\begin{defi}
Let $D$ be a effective divisor on  $X$. The degree of $D$ is defined
by $deg(\mathscr{O}(D))$.
\end{defi}

\begin{obs}
Since $D$ is effective we have that
$H^{0}(X,\mathscr{O}(D))\neq\{0\}$, and follows from theorem
\ref{posi} that $deg(\mathscr{O}(D))>0$.
\end{obs}

Let $D$ be a divisor on  $X$ defined locally by functions
$\{f_{\alpha}\in \mathscr{O}(\mathcal{U}_{\alpha})\}_{\in\Lambda}$,
where $\{\mathcal{U}_{\alpha}\}_{\in\Lambda}$ is a open covering of
$X$.  If
$\mathcal{U}_{\alpha\beta}:=\mathcal{U}_{\alpha}\cap\mathcal{U}_{\beta}$
then there exist $f_{\alpha\beta}\in
\mathscr{O}^*(\mathcal{U}_{\alpha})$, such that
$f_{\alpha}=f_{\alpha\beta}f_{\beta}$. Denote by $f_{\alpha}^{D}$
the restriction of $f_{\alpha}$ on $D$. Let $\F$ be a holomorphic
foliation given by collections
$(\{\vartheta_{\alpha}\};\{\mathcal{U}_{\alpha}\};
\{g_{\alpha\beta}\in \mathcal{O}_{_{\mathcal{U}_{\alpha}}}^*
;\})_{\alpha \in \Lambda}$ on $X$. Consider the following functions
$$\zeta_{\alpha}^{(\F,D)}=\vartheta_{\alpha}(f_{\alpha}^{D})\in \mathscr{O}(\mathcal{U}_{\alpha}\cap D).$$
If $\mathcal{U}_{\alpha}\cap\mathcal{U}_{\beta}\cap D\neq\emptyset$
and using the Leibniz's rule we get
$\zeta_{\alpha}^{(\F,D)}=f_{\alpha\beta}^{D}g_{\alpha\beta}\zeta_{\beta}^{(\F,D)}$.
With this , we obtain a global section $\zeta^{(\F,D)}$ of line
bundle $(T_{\F^*} \otimes [D])_{\mid D}$. The \emph{tangency
varietie} of $\F$
  with $D$ is given by
  $$
\mathcal{T}(\F,D)=\{p\in D; \zeta^{(\F,D)}(p)=0\}.
  $$
\begin{defi}
Let $X\subset \mathbb{P}^{N}$ be a smooth algebraic variety and $H$
the hyperplane class of $\mathbb{P}^{N}$. Let $\F$ be a foliation on
$X$. The degree of $\F$ is the intersection number
$$
deg(\F):=\langle[\mathcal{T}(\F,H)]\smile [H]^{(n-2)},[H]\rangle,
$$
where $[H]^{(n-2)}=\underbrace{[H]\smile\cdots\smile
[H]}_{(n-2)-times}.$
\end{defi}

\begin{prop}\label{teorema}
Let $\F$ be a foliation on algebraic variety $X\subset\mathbb{P}^N$.
Then
$$
deg(\F)=deg(T_{\F}^*)+deg(X),
$$
where $deg(X)$ is the degree of $X$.
\end{prop}

\begin{proof}
 We have that $$\langle[\mathcal{T}(\F,H)]\smile
[H]^{(n-2)},[H]\rangle=\displaystyle\int\limits_{H}c_1([\mathcal{T}(\F,H)])\wedge
  h^{n-2},
$$
where $h$ is the hyperplane class. By adjunction formula
$[\mathcal{T}(\F,H)]=(T_{\F^*} \otimes [H])_{\mid H}$ and since $H$
is Poincaré's dual of $h$, that is $c_1\left([H]\right)=h$, we get
the following
$$
\begin{array}{ccl}
   deg(\F)&=&\langle[\mathcal{T}(\F,\mathcal{H})]\smile
[H]^{(n-2)},[H]\rangle=\displaystyle\int_{H}c_1(T_{\F^*} \otimes
[H])_{\mid
H})\wedge h^{n-2}\\
  \\
   &= &\displaystyle\int_{H}c_1\left(T_{\F_{|_{H}}}^*
 \right)\wedge h^{n-2}+\int_{H}c_1\left([H]\right)\wedge h^{n-2}\\
 \\
   & =&\displaystyle\int_{X}c_1\left(T_{\F}^*
 \right)\wedge h^{n-1}+\int_{X}h \wedge h^{n-1}\\
 \\
  & =&deg\left( T_{\F}^*\right)+deg(X).
\end{array}
$$

\end{proof}

\begin{obs}
If $deg\left( T_{\F}^*\right)<0$ follows from theorem \ref{posi}
that $H^0(X,T_{\F}^*)=\{0\}$. Therefore we shall assume $deg\left(
T_{\F}^*\right)> 0$, or equivalently $deg(\F)-deg(X)>0.$

\end{obs}

\begin{exe}
Let $\F$  be a foliation on $X$,  where $Pic(X)\simeq\mathbb{Z}$.
Take a positive generator $\mathcal{H}$ for $Pic(X)$ and denote by
$\mathcal{O}_{X}(k):=\mathcal{H}^{\otimes k}$ the $k$-th tensorial
power of $\mathcal{H}$. If we shall write
$T_{\F}^*=\mathcal{O}_{X}(d-1)$ we get that
$deg(T_{\F}^*)=(d-1)deg(X)$ and hence
$$
deg(\F)=deg(T_{\F}^*)+deg(X)=(d-1)deg(X)+deg(X)=d\cdot deg(X).
$$
In the case where $X=\mathbb{P}^n$ we will have, as already it is
known, that  $deg(\F)=d$.
\end{exe}

\section{Extatic divisor}
The method adopted here stems from the work of J.V.Pereira [P],
where the notion of extactic variety is exploited. In this section
we  digress briefly on extactic varieties and their main properties.

A one-dimensional foliation $\F$ on complex manifold $X$ induced a
morphism $\Phi_{\F}: \Omega_{X}^1\rightarrow T_{\F}^*$ given by
locally by contraction , that is,
$\Phi_{\F_{|_{\mathcal{U}_{\alpha}}}}(\theta)=i_{\vartheta_{\alpha}}(\theta_{\alpha})$,
where $\mathcal{U}_{\alpha}$ is a opened of $X$.

Consider the linear system $H^{0}(X,\mathcal{O}(D))$ and  take a
open covering  $\{\mathcal{U}_{\alpha}\}_{\alpha}$ of $X$ which
trivialize  $\mathcal{O}(D)$ and $T_{\F}^*$. In the opened
$\mathcal{U}_{\alpha}$ we can consider the morphism
$$
T^{(k)}_{|_{\mathcal{U}_{\alpha}}}:H^{0}(X,\mathcal{O}(D))\otimes
\mathcal{O}_{\mathcal{U}_{\alpha}}\rightarrow
\mathcal{O}_{\mathcal{U}_{\alpha}}^{k}
$$
defined by
$$
T^{(k)}_{|_{\mathcal{U}_{\alpha}}}(s)=s+X_{\F}(s)^{\alpha}\cdot
t+X_{\F}^2(s)^{\alpha}\cdot
\frac{t^2}{2!}+\cdots+X_{\F}^k(s)^{\alpha}\cdot \frac{t^k}{k!},
$$
where
$X_{\F}(\cdot)^{\alpha}=\Phi_{\F}(d(\cdot))_{|_{\mathcal{U}_{\alpha}}}$
and $s\in H^{0}(X,\mathcal{O}(D))\otimes
\mathcal{O}_{\mathcal{U}_{\alpha}}.$ In an opedend
$\mathcal{U}_{\alpha}$ we have
$\mathcal{O}(D)_{|_{\mathcal{U}_{\alpha}}}=\mathcal{O}_{\mathcal{U}_{\alpha}}\cdot
\sigma_{\alpha}$ and
$T_{\F_{|_{\mathcal{U}_{\alpha}}}}^*=\mathcal{O}_{\mathcal{U}_{\alpha}}\cdot
\beta_{\alpha}$. Therefore, for all $s_{\alpha} \in
H^{0}(X,\mathcal{O}(D))\otimes \mathcal{O}_{\mathcal{U}_{\alpha}}$
we obtain
$$
\begin{array}{c}
  s_{\alpha}=s_{\alpha}^{(1)}\cdot \sigma_{\alpha} \\
  X_{\F}(s_{\alpha})^{\alpha}= X_{\F}(s_{\alpha}^{(1)})^{\alpha} \cdot
  \beta_{\alpha}=s_{\alpha}^{(2)}\cdot \beta_{\alpha}
  \\
  \vdots
  \\
  X_{\F}^{k-1}(s_{\alpha})^{\alpha}= X_{\F}(s_{\alpha}^{(k-2)})^{\alpha} \cdot
  \beta_{\alpha}=s_{\alpha}^{(k)}\cdot \beta_{\alpha}
\end{array}
$$
If $\mathcal{U}_{\alpha}\cap \mathcal{U}_{\gamma}\neq\emptyset$ then
$s_{\alpha}^{(1)}=g_{\alpha\gamma} s_{\gamma}^{(1)}$ and
$X_{\F}(\cdot)^{\alpha}=i_{\vartheta_{\alpha}}(\cdot)=i_{(f_{\alpha\gamma}\vartheta_{\gamma})}(\cdot)=
f_{\alpha\gamma}X_{\F}(\cdot)^{\gamma}$, where
$g_{\alpha\gamma},f_{\alpha\gamma} \in
\mathcal{O}^*(\mathcal{U}_{\alpha})$ are the cocycles which defines,
respectively, the line bundles  $[D]$ and $T_{\F}^*$. using the
described compatibility above and the Leibniz's rule we get
$$
\begin{array}{c}
  s_{\alpha}=s_{\alpha}^{(1)}\cdot \sigma_{\alpha}=g_{\alpha\beta} s_{\gamma}^{(1)}\cdot \sigma_{\alpha}\\
  X_{\F}(s_{\alpha})^{\alpha}= X_{\F}(s_{\alpha}^{(1)})^{\alpha} \cdot
  \beta_{\alpha}=(X_{\F}(g_{\alpha\gamma})^{\gamma}\cdot
  s_{\gamma}^{(1)}+g_{\alpha\gamma}\cdot s_{\gamma}^{(2)} )\cdot f_{\alpha\gamma}
  \cdot\beta_{\gamma}
\end{array}
$$
Following for this process it ties the order
$k=h^{0}(X,\mathcal{O}(D))$, we obtain
$$
\begin{array}{ccc}
  \left[
     \begin{array}{c}
       s_{\alpha}^{(1)}\\
       s_{\alpha}^{(2)} \\
       s_{\alpha}^{(3)} \\
       \vdots\\
        s_{\alpha}^{(k)}\\
     \end{array}
   \right]&=
   & \left[
       \begin{array}{ccccc}
         g_{\alpha\beta} & 0 & 0 & 0 & 0\\
          X_{\F}(g_{\alpha\gamma})^{\gamma}\cdot f_{\alpha\gamma}& g_{\alpha\beta}\cdot f_{\alpha\beta}& 0 & 0 & 0 \\
         \ddots & \ddots & g_{\alpha\beta}\cdot f_{\alpha\beta}^2& 0 & 0 \\
         \ddots & \ddots & \ddots & \ddots & 0 \\
         \ddots & \ddots & \ddots & \ddots & g_{\alpha\beta}\cdot f_{\alpha\beta}^{k-1} \\
       \end{array}
     \right]\cdot\left[
     \begin{array}{c}
       s_{\gamma}^{(1)}\\
       s_{\gamma}^{(2)} \\
       s_{\gamma}^{(3)} \\
       \vdots\\
        s_{\gamma}^{(k)}\\
     \end{array}
   \right]
\end{array}
$$
Denoting the matrix above by  $\Theta_{\alpha\gamma}(\F,D)\in
GL(k,\mathcal{O}^*_{\mathcal{U}_{\alpha\gamma}})$, we see that
$$
\left\{
\begin{array}{ll}
\Theta_{\alpha\gamma}(\F,D)(p)\cdot\Theta_{\gamma\alpha}(\F,D)(p)=I,\ \  $for all$ \ p\in \mathcal{U}_{\alpha}\cap \mathcal{U}_{\gamma}\\
\\
\Theta_{\alpha\gamma}(\F,D)(p)\cdot\Theta_{\gamma\lambda}(\F,D)(p)\cdot\Theta_{\lambda\alpha}(\F,D)(p)=I,\
\ $for all$
\ p\in \mathcal{U}_{\alpha}\cap \mathcal{U}_{\gamma}\cap \mathcal{U}_{\lambda}.\\
\end{array}
\right.
$$
That is, the family  of matrices
$\{\Theta_{\alpha\gamma}(\F,D)\}_{\alpha\gamma}$ define a cocycle of
a vector bundle of rank $k$ on $X$ that we shall denote by
$J_{\mathcal{X}_{\F}}^{k}\mathcal{O}(D)$. Now, using the
trivializations $\{\Theta_{\alpha\gamma}(\F,D)\}_{\alpha\gamma}$ we
get the morphisms
$$
T^{(k)}:H^{0}(X,\mathcal{O}(D))\otimes \mathcal{O}_{X}\rightarrow
J_{\mathcal{X}_{\F}}^{k}\mathcal{O}(D).
$$
Taking the determinant  of $T^{(k)}$ we have the morphism
\begin{center}
$det(T^{(k)}):\bigwedge^k\left[H^{0}(X,\mathcal{O}(D))\right]\otimes
\mathcal{O}_{X}\rightarrow
\bigwedge^kJ_{\mathcal{X}_{\F}}^{k}\mathcal{O}(D)$,
\end{center}
and tensorizing by $(\bigwedge^kV)^*$ we obtain a global section of
$\bigwedge^kJ_{\mathcal{X}_{\F}}^{k}\mathcal{O}(D)\otimes
(\bigwedge^kV)^*$ given by
\begin{center}
$\varepsilon_{_{(\F,V)}}: \mathcal{O}_{X}\rightarrow
\bigwedge^kJ_{\mathcal{X}_{\F}}^{k}\mathcal{O}(D)\otimes
(\bigwedge^kV)^*$.
\end{center}

\begin{defi}
The Extatic divisor  of $\F$ with respect to the linear system $
H^{0}(X,\mathcal{O}(D))$ is the divisor
$\mathcal{E}(\F,V)=\left(\varepsilon_{_{(\F,V)}}\right)$ given by
zeros of  the section  $$\varepsilon_{_{(\F,V)}} \in
H^{0}\left(X,\bigwedge^kJ_{\mathcal{X}_{\F}}^{k}\mathcal{O}(D)\otimes
(\bigwedge^kV)^*\right).$$
\end{defi}

J.V.Pereira [P] obtained the following results, which elucidate the
role of the divisor variety:

\begin{prop}([P], Proposition 5)\label{prop}Let $\F$ be a one-dimensional holomorphic foliation
on a complex manifold $X$. If $V$ is a finite dimensional linear
system, then every $\F$- invariant hypersurface which is contained
in the zero locus of some element of V , must be contained in the
zero locus of $\mathcal{E}(V,\F)$.
\end{prop}

If $\F$ is a holomorphic one-dimensional foliation on a complex
manifold $X$, then a first integral for $\F$ is a holomorphic map
$\Theta:X\longrightarrow Y$, where $Y$ is a complex manifold, such
that the fibers of $\Theta$ are $\F$-invariant. Then we have:

\begin{teo}([P], Theorem 3)\label{jorge}. Let $\F$ be a one-dimensional holomorphic
foliation on a complex manifold $X$. If $V$ is a finite dimensional
linear system such that $\mathcal{E}(V,\F)$ vanishes identically,
then there exits an open and dense set $U$ where $\F_{|U}$ admits a
first integral. Moreover, if $X$ is a projective variety, then $\F$
admits a meromorphic first integral.
\end{teo}

\section{Proofs}

\subsection{Proof of theorem \ref{teo}
} From theorem  \ref{jorge} if $\F$ does not rational first integral
$\varepsilon_{_{(\F,V)}}\neq 0$, and then defines a divisor
$\mathcal{E}(\F,V)$ whose line bundle associated is
$\bigwedge^kJ_{\mathcal{X}_{\F}}^{k}\mathcal{O}(D)\otimes
(\bigwedge^kV)^*$. Let us say that $k=dim_{\mathbb{C}}V$. Let
$\mathscr{N}_i$ be the number of irreducible divisors  of
$H^{0}(X,\mathcal{O}(D))$ of degree  $i\leq deg(D)$, counting
multiplicities, invariants by $\F$. From proposition \ref{prop} all
divisor $\mathscr{D}  \in H^{0}(X,\mathcal{O}(D))$  invariant by
$\F$ is contained in the extatic  $\mathcal{E}(\F,V)$. Using this
fact we can to  affirm that
$$
\sum_{i=1}^{deg(D)}i\cdot\mathscr{N}_i\leq deg(\mathcal{E}(\F,V)).
$$
Indeed, it is enough to group the divisors $\F$-invariants of the
following form
$$
[\mathcal{E}(\F,V)]=[V_{1}^1]^{d_{_{11}}}\otimes\cdots \otimes
[V^{n_1}_1]^{d_{_{1n_1}}}\otimes \cdots
\otimes[V^{1}_{deg(D)}]^{d_{_{1deg(D)}}} \otimes\cdots
$$
$$
\cdots \otimes
[V^{n_{deg(D)}}_{deg(D)}]^{d_{_{n_{deg(D)deg(D)}}}}\otimes\mathcal{L},
$$
where $[V_{i}^j]$ is a divisor irreducible  invariant by $\F$, of
degree $i$ and  multiplicities  $d_{_{ij}}$, and $\mathcal{L}$ is a
line bundle. Therefore we get
$$
\sum_{k=1}^{n_i}d_{_{ki}}deg(V_{k}^i)=i\cdot\sum_{k=1}^{n_i}d_{_{ki}}=i\cdot\mathscr{N}_i,\
\ \forall \ i=1,\dots deg(D).
$$
For simplicity we will write
$[\mathcal{E}(\F,V)]=\mathfrak{I}_{\F}\otimes \mathcal{L}$, where
$$
\mathfrak{I}_{\F}=[V_{1}^1]^{d_{_{11}}}\otimes\cdots \otimes
[V^{n_1}_1]^{d_{_{1n_1}}}\otimes \cdots
\otimes[V^{1}_{deg(D)}]^{d_{_{1deg(D)}}} \otimes\cdots \otimes
[V^{n_{deg(D)}}_{deg(D)}]^{d_{_{n_{deg(D)deg(D)}}}}.
$$
Calculating the degree we conclude that
$deg(\mathfrak{I}_{\F})=\displaystyle
\sum_{i=1}^{deg(D)}i\cdot\mathscr{N}_i\leq deg(\mathcal{E}(\F,V))$.
This show the affirmed one above.
\\
\\
Fronm this inequality, we get the following  $
deg(D)\cdot\mathscr{N}(D)\leq deg(\mathcal{E}(\F,V)).$ However the
line bundle associated to the extatic divisor $\mathcal{E}(\F,V)$ is
given by the following
 $\bigwedge^kJ_{\mathcal{X}_{\F}}^{k}\mathcal{O}(D)\otimes
(\bigwedge^kV)^*.$ This implies that $$ [\mathcal{E}(\F,V)]\simeq
\bigwedge^kJ_{\mathcal{X}_{\F}}^{k}\mathcal{O}(D)\otimes
(\bigwedge^kV)^*.
$$
On the other hand, the cocycle of
$\bigwedge^kJ_{\mathcal{X}_{\F}}^{k}\mathcal{O}(D)$ is given by
$$det(\Theta_{\alpha\gamma}(\F,D))=g_{\alpha\beta}^k \cdot
f_{\alpha\beta}^{{k \choose 2}}, $$ where $g_{\alpha\beta}$ and
$f_{\alpha\beta}$ are trivializations of $[D]$ and $T_{\F}^*$,
respectively. This show that
$\bigwedge^kJ_{\mathcal{X}_{\F}}^{k}\mathcal{O}(D)\simeq
[D]^{\otimes k}\otimes (T_{\F}^*)^{\otimes {k \choose 2}},$ hence
\begin{center}
$[\mathcal{E}(\F,V)]= [D]^{\otimes k}\otimes (T_{\F}^*)^{\otimes {k
\choose 2}}\otimes (\bigwedge^k V)^*$.
\end{center}
Calculating the degree $deg(\mathcal{E}(\F,V))$ we get
$$deg(\mathcal{E}(\F,V))=deg\left([D]^{\otimes k}\otimes
(T_{\F}^*)^{\otimes {k \choose 2}}
\right)+\underbrace{deg\left(\bigwedge^k V^*
\right)}_{\stackrel{\shortparallel}{0} }=k\cdot
deg(D)+deg(T_{\F}^*){k \choose 2}.$$ Finalely, follows from
$\mathscr{N}_{deg(D)}\cdot deg(D)\leq\displaystyle
\sum_{i=1}^{deg(D)}i\cdot\mathscr{N}_i\leq deg(\mathcal{E}(\F,V))$
and of the fact that $deg(T_{\F}^*)=deg(\F)-deg(X)$, that
$$
deg(D)\cdot[\mathscr{N}_{deg(D)}-k] \leq [deg(\F)-deg(X)]\cdot{k
\choose 2}.
$$
This proof the theorem \ref{teo}.

\subsubsection{Proof of corollary \ref{cor}}
From Riemann-Roch's theorem (see \cite{H} theorem 1.6) we get
$$
h^0(D)=h^1(D)-h^0(K-D)+\frac{D\cdot(D-K)}{2}+\chi(\mathscr{O}_{X}),
$$
where $\chi(\mathscr{O}_{X})$ holomorphic Euler characteristic of
$X$ and $K$ is the canonical sheave. Since
$g(X,D)=\frac{D\cdot(D-K)}{2}+D\cdot K+1$ we have that
$$
h^0(D)=h^1(D)-h^0(K-D)+g(X,D)-D\cdot K+\chi(\mathscr{O}_{X})-1.
$$
From Noether's formula $\chi(\mathscr{O}_{X})=\frac{1}{12}(K\cdot
K+\chi(X))$  we get
$$
h^0(D)=h^1(D)-h^0(K-D)+g(X,D) +\frac{1}{12}[K\cdot(
K-12D)+\chi(X)]-1\ \ (*)
$$
Now, by theorem \ref{teo}  we have that
$$
\begin{array}{ccc}
 - h^0(D) & \leq& \left(\frac{deg(\F)-deg(X)}{deg(D)}\right)\cdot{h^0(D)
\choose 2}-\mathscr{N}(\F,|D|) \\
\end{array}
$$
The result follows from this inequality and $(*)$.
\\
\\
\textbf{Acknowledgement}:
\\
\\
I would like to be thankful to Marcio G. Soares for to guide my work
and Rodrigo Bissacot for interesting conversations.


{\footnotesize
\begin{thebibliography}{99}
\bibitem[V]{V} J. V. Pereira, \emph{Vector Fields, Invariant Varieties and Linear Systems}.
Annales de L\'Institut Fourier 51, no.5 (2001), 1385-1405.

\bibitem[CN]{CN}
D. Cerveau and A. Lins Neto, \emph{Holomorphic foliations in
$\mathbb{P}_{\mathbb{C}}^2$ having an invariant algebraic curve},
Ann. Inst. Fourier \textbf{41} (1991), 883-903.

\bibitem[Z]{Z}
A. G. Zamora, \emph{Foliations in Algebraic Surfaces having a
rational first integral}, Publicacions Matematiques \textbf{41}
(1997), 357-373.

\bibitem[C-L]{C-L} V. Cavalier and D. Lehmann, \emph{On the Poincar´e inequality for
one-dimensional foliations}, Com. Compositio Math. \textbf{142}
(2006) 529-540.

\bibitem[N]{N}
A. Lins Neto, \emph{Some examples for Poincaré and Painleve problem}
Ann. Scient. Ec. Norm. Sup., 4e série, \textbf{35}, 2002, p. 231 a
266.

\bibitem[O]{O}J. M. Ollagnier,\emph{About a conjecture on quadratic vector
fields},Journal of Pure and Applied Algebra \textbf{165} (2001)
227-234.

\bibitem[S]{S} M. G. Soares, \emph{The Poincaré problem for hypersurfaces invariant by
one-dimensional foliations}, Inventiones Mathematicae, Alemanha, v.
\textbf{128}, p. 495-500, 1997.

\bibitem[E-K]{E-K} E. Esteves and S. Kleiman, \emph{Bounds on leaves of one-dimensional
foliations}, Bull. Braz. Mat. Soc. (NS) \textbf{34} (2003),145-169.
\bibitem[H]{H}
R. Hartshorne, \emph{Algebraic Geometry}, Springer-Verlag Graduate
Texts in Mathematics \textbf{52}, 1977.

\bibitem[K]{K}
S. Kobayashi, \emph{Differential Geometry of complex Vector Bundle};
Publication of the Mathematical Society of Japan, Princeton
University Press, 1987.

\bibitem[C]{C}
M. Carnicer, \emph{The Poincaré problem in the non-dicritical case},
Ann. de Math. \textbf{140} (1994) 289-294.

\bibitem[HP]{HP}
H. Poincaré, \emph{Sur l'integration algébrique des équations
differéntielles du premier order et du premier degré} Rend. Circ Mat
Palermo,\textbf{5} (1891), 161-191.

\bibitem[PP]{PP}P. Painlevé, \emph{Sur les intégrales algébrique des équations
differentielles du premier ordre} and \emph{Mémoire sur les
équations différentielles du premier ordre}, Oeuvres de Paul
Painlevé; Tome II, Éditions du Centre National de la Recherche
Scientifique, \textbf{15}, quai Anatole-France, 75700, Paris, 1974.

\bibitem[B-M]{B-M}
M. Brunella and L.G. Mendes, \emph{Bounding the degree of solutions
to Pfaff equations}, preprint \textbf{206}, U. Bourgogne, 1999.



\end{thebibliography}}
\end{document}